\documentclass[11pt]{article}   
\usepackage[greek,english]{babel}  
\usepackage{graphicx}
\usepackage[pdftex,bookmarks]{hyperref}
\usepackage{amsmath}
 \usepackage{amssymb}
\usepackage{amsmath,amsthm}
\usepackage{verbatim}
\usepackage{sidecap}
\usepackage{tikz}
\newtheorem{theorem}{Theorem}[section]

\newtheorem{lemma}[theorem]{Lemma}

\theoremstyle{definition}
\newtheorem{remark}[theorem]{Remark}

\newcommand{\RR}{\mathbb{R}}

\title{Positive solutions to indefinite Neumann problems
\\when the weight
has positive average}
\author{Alberto Boscaggin and Maurizio Garrione}
\date{}

\begin{document}
\maketitle

\begin{abstract}
We deal with positive solutions for the Neumann boundary value problem associated with
the scalar second order ODE
$$
u'' + q(t)g(u) = 0, \quad t \in [0, T],
$$
where $g: [0, +\infty[\, \to \RR$ is positive on $\,]0, +\infty[\,$ and $q(t)$ is an indefinite weight. Complementary to
previous investigations in the case $\int_0^T q(t) < 0$, we provide existence results for a suitable class of weights having (small) positive mean, when $g'(x) < 0$ at infinity. Our proof relies on a shooting argument for a suitable equivalent planar system of the type
$$
x' = y, \qquad y' = h(x)y^2 + q(t),
$$
with $h(x)$ a continuous function defined on the whole real line.
\end{abstract}

\noindent
{\footnotesize \textbf{AMS-Subject Classification}}. {\footnotesize 34B15; 34B08.}\\
{\footnotesize \textbf{Keywords}}. {\footnotesize Indefinite weight;
Average condition; Neumann problem; Shooting method.}

\section{Introduction and statement of the main result}

In this paper, we deal with the existence of \emph{positive} solutions to the Neumann problem 
\begin{equation}\label{eqintro}
\left\{
\begin{array}{l}
u'' + q(t) g(u) = 0, \qquad t \in [0, T], \vspace{0.2 cm}\\
u'(0) = u'(T) = 0, 
\end{array}
\right.
\end{equation}
where $q \in L^1(0,T)$ and $g: [0,+\infty[\, \to \mathbb{R}$ is a $C^1$-function satisfying 
\begin{equation}\label{fondamentale}
g(0)=0 \quad \mbox{ and } \quad g(u) > 0, \quad \mbox{for every } u > 0.
\end{equation}
Integrating the equation on $[0, T]$, we immediately see that no positive solutions appear if the weight function $q(t)$ has constant sign; hence, we are led to take into account only \emph{indefinite} weights. This terminology, meaning that $q(t)$ takes both positive and negative values, is nowadays quite common
in Nonlinear Analysis, and it probably goes back to the papers \cite{AtkEveOng74,HesKat80}.
\smallbreak
A further condition on $q(t)$ naturally appears when dealing with problem \eqref{eqintro}, as already observed, for instance, in 
\cite{BanPozTes88,BosZan12,BosZan15}. Indeed, dividing the equation by $g(u(t))$ and integrating on $[0,T]$,
we find
\begin{equation}\label{med}
\int_0^T q(t)\,dt = -\int_0^T \left(\frac{u'(t)}{g(u(t))}\right)^2 g'(u(t))\,dt.
\end{equation}
Hence, the existence of positive solutions cannot be ensured if the weight function has nonnegative mean value whenever 
$g'(u) > 0$ for any $u > 0$ (this is the case, for instance, for the model nonlinearity $g(u) = u^p$ with $p > 0$).
Incidentally, we notice that the same considerations hold true for the $T$-periodic problem, as well as for the PDE counterpart of problem \eqref{eqintro}, as a consequence of the divergence theorem. 
\smallbreak
There is a rich bibliography showing that, under suitable supplementary assumptions on the nonlinear term
$g(u)$ and on the weight function $q(t)$, the average condition 
\begin{equation}\label{negativa}
\int_0^T q(t) \, dt < 0
\end{equation}
turns out to be sufficient to ensure the existence of positive solutions to \eqref{eqintro},
possibly also in the case when $g \notin C^1$ or $g'(u)$ changes its sign.
The first result in this direction was given by Bandle, Pozio and Tesei in \cite{BanPozTes88}, providing the solvability of 
the Neumann problem associated with the elliptic equation $\Delta u + q(x)g(u) = 0$
in the sublinear case, i.e., when $g(u)$ behaves like $u^p$ both at zero and at infinity, with $0 < p < 1$. Later on, several statements for the superlinear case (that is, $g(u)$ like
$u^p$ both at zero and at infinity with $p > 1$) followed, both for elliptic Neumann problems and Neumann or periodic ODE ones (see, among others, \cite{AlaTar93,AmaLop98,BerCapNir94,BerCapNir95,BosZan15,FelZan15} and the references therein). In all these contributions, the average condition \eqref{negativa} (or the corresponding one for the PDE case) plays a crucial role. 
\smallbreak
More recently, some attention has also been devoted to the so-called \emph{super-sublinear} case, meaning that
$g(u)$ behaves like $u^{\alpha}$, with $\alpha > 1$, at zero, and like $u^{\beta}$, with $\beta < 1$, at infinity; more in general, 
\begin{equation}\label{supersub}
\frac{g(u)}{u} \to 0 \quad \mbox{ both for } u \to 0 \;\mbox{ and } u \to +\infty.
\end{equation}
In this case, on the lines of classical works by Amann \cite{Ama72} and Rabinowitz \cite{Rab73},
it is convenient to write the weight function $q(t)$ in
\eqref{eqintro} as $q(t) = q_{\lambda}(t) = \lambda a(t)$, with $\lambda > 0$,
thus dealing with a parameter dependent equation like
\begin{equation}\label{supsub}
u'' + \lambda a(t) g(u) = 0.
\end{equation}
Up to some additional (mild) technical assumptions, a pair of positive Neumann solutions is then obtained when $\int_0^T q_{\lambda}(t)\,dt = \lambda \int_0^T a(t)\,dt < 0$ and $\lambda$ is big enough \cite{BosFelZan16,BosZan12}. The largeness of the parameter is indeed essential, since nonexistence results can often be provided for $\lambda$ small (see again the previously quoted papers and Remark \ref{parametri}). 
\smallbreak
It is a natural question, on the contrary, whether there are examples of Neumann problems of the class
\eqref{eqintro} admitting a positive solution when
$$
\int_0^T q(t) \, dt \geq 0;
$$
of course, in view of the above discussion,
this is possible only if $g'(u)$ changes its sign on $]0,+\infty[\,$.
To the best of our knowledge, this issue is far less investigated and it is the aim of the present paper to give a possible contribution in this direction.  
Precisely, we will consider the case when there exists $R > 0$ such that 
\begin{equation}\label{ipoinf}
g'(u) < 0 \; \textnormal{ for every } u > R \quad \textnormal { and } \quad \lim_{u \to +\infty} g'(u)=0. 
\end{equation}
This of course implies that $g(u)$ is bounded (decreasingly converging to a horizontal asymptote), thus entering the setting of problems with sublinear growth at infinity. As for the conditions at zero, motivated by the existing literature, we choose to assume that there exists $r > 0$ such that 
\begin{equation}\label{ipozero}
g'(u) > 0 \;\textnormal{ for every } 0 <u < r \quad \textnormal { and } \quad \lim_{u \to 0^+} g'(u)=0. 
\end{equation}
Summing up, we thus deal with a particular type of super-sublinear problem (indeed, condition \eqref{supersub} is of course satisfied); an explicit example of nonlinearity satisfying 
\eqref{fondamentale}, \eqref{ipoinf} and \eqref{ipozero} and thus fitting in our framework is given by
$$
g(u) = \frac{c_1 u^{\alpha}}{1 + c_2 u^{\gamma}}
$$
for $c_1, c_2 > 0$ and $1 < \alpha < \gamma$. 
\smallbreak
Roughly speaking, we are going to show that, whenever \eqref{ipoinf} and \eqref{ipozero} hold, the existence of positive solutions
to \eqref{eqintro} is still ensured if the average of the weight function is ``positive but small enough''. To express this in a formal way, we 
need to introduce - with respect to \eqref{supsub} - a further parameter $\mu > 0$, thus writing the weight function in the form
$$
q(t) = q_{\lambda,\mu}(t) = \lambda a^+(t) - \mu a^-(t),
$$ 
with $a^+(t), a^-(t)$ the positive and the negative part of a sign-changing 
function $a \in L^1(0, T)$ (accordingly, $a(t)$ is sign-changing up to zero measure sets, i.e., $\int_0^T a^+(t) \, dt
, \int_0^T a^-(t) \, dt > 0$). We also make the further assumption that
$a(t)$ changes sign just once on $[0,T]$, namely (up to substituting $T-t$ for $t$) there exists $\tau \in [0, T]$ such that  
\begin{equation}\label{ipotesia}
a(t) \geq 0 \; \mbox{ for a.e. } t \in [0, \tau], \qquad a(t) \leq 0 \; \mbox{ for a.e. } t \in [\tau, T].
\end{equation}
Moreover, for any $\lambda > 0$ we define 
\begin{equation}\label{mucritico}
\mu_0(\lambda) = \lambda\;\frac{\int_0^T a^+(t) \, dt}{\int_0^T a^-(t) \, dt}\,,
\end{equation}
in such a way that
$$
\int_0^T q_{\lambda,\mu}(t)\,dt \geq 0 \quad \Longleftrightarrow \quad \mu \leq \mu_0(\lambda).
$$
We are now in a position to illustrate the complete picture concerning positive solutions to the Neumann problem
\begin{equation}\label{eqintro2}
\left\{
\begin{array}{l}
u'' + \big( \lambda a^+(t)- \mu a^-(t)\big) g(u) = 0, \qquad t \in [0, T], \vspace{0.2 cm}\\
u'(0) = u'(T) = 0, 
\end{array}
\right.
\end{equation}
giving the following statement.

\begin{theorem}\label{thmain}
Assume that \eqref{fondamentale}, \eqref{ipoinf}, \eqref{ipozero} and \eqref{ipotesia} hold. Then, for any $\lambda > 0$ there exist $\mu_-(\lambda), \mu_+(\lambda) > 0$, with $\mu_-(\lambda) < \mu_0(\lambda) < \mu_+(\lambda)$, such that problem \eqref{eqintro2} has at least 
one positive solution for $\mu=\mu_0(\lambda)$ and at least two positive solutions for any $\mu \in \,]\mu_-(\lambda), \mu_+(\lambda)[\,$
different from $\mu_0(\lambda)$. Moreover, $\mu_+(\lambda) = +\infty$
for any $\lambda$ sufficiently large, that is, for $\lambda > \lambda^*$, with $\lambda^* > 0$ depending only on $g(u)$ and
$a^+(t)$.
\end{theorem}
Some part of this statement is not new. Precisely, the existence of two positive solutions for $\lambda$ large enough and any $\mu > \mu_0(\lambda)$
(recall that in this case $\mu_+(\lambda)=+\infty$) is a corollary\footnote{Indeed, in \cite[Theorem 1.2]{BosFelZan16} the existence of two positive $T$-periodic solutions is proved, for $\hat{\lambda}$ larger than a suitable constant $\hat{\lambda}^* > 0$,
for the equation
$u'' + \hat{\lambda} \hat{a}(t)g(u) = 0$, whenever $g(u)$ satisfies \eqref{fondamentale} as well as $g'(0) = g'(\infty) = 0$
and $\hat{a}(t)$ is positive (and not identically zero) on some subinterval $I \subset [0,T]$ with $\int_0^T \hat{a}(t)\,dt < 0$.
In the discussion leading to \cite[Corollary 1.1]{BosFelZan16}, moreover, it was shown that $\hat{\lambda}^*$ can be chosen depending only on the behaviour of $\hat{a}(t)$ on $I$; this yields the existence of two positive $T$-periodic solutions for the equation $u'' + \big( \lambda a^+(t)- \mu a^-(t)\big) g(u) = 0$
when $\lambda$ is large enough and $\mu > \mu_0(\lambda)$. The possibility of taking into account also Neumann boundary conditions
is discussed in \cite[Section 5]{BosFelZan16}.} of the results in \cite{BosFelZan16}. However, Theorem \ref{thmain} refines this piece of information, since a pair of positive solutions appears \emph{for every} $\lambda > 0$ (maybe unexpectedly; we will briefly comment on this in Remark \ref{parametri}), provided that the average of the weight function is ``negative but small enough'', that is, $\mu \in \,]\mu_0(\lambda), \mu_+(\lambda)[\,$. On the other hand, as desired, Theorem \ref{thmain} gives the existence of positive solutions 
to \eqref{eqintro2} when the average of $q_{\lambda,\mu}(t)$ is ``positive but small enough'', namely for 
$\mu \in \,]\mu_-(\lambda),\mu_0(\lambda)]$; again, this happens for any $\lambda > 0$.
Of course, we notice that, given an indefinite weight $q(t)$ in \eqref{eqintro} with positive mean value, by no means we can establish, via Theorem \ref{thmain}, if its average is so small that the existence is ensured; however, our result certainly allows to find many examples of weights with (small) positive mean for which this happens and, even more, this can be done simply by suitably contracting the negative part of any sign-changing weight $a \in L^1(0,T)$ (satisfying \eqref{ipotesia}).  

We refer to Figure \ref{pianolambdamu} at the end of the paper for a visual representation of the statement, while
its optimality (referring to the precise meaning
of the expressions ``negative/positive but small enough'') will be discussed in Remark \ref{Nonesiste}.
\smallbreak
To prove Theorem \ref{thmain}, we will use a dynamical argument relying on an elementary shooting method. More precisely, through the change of variables introduced in \cite{BosZan15}, we transform the equation in \eqref{eqintro2} into a first order planar system of the type
\begin{equation}\label{sistema}
\left\{
\begin{array}{l}
x' = y  \vspace{0.2 cm} \\
y' = h(x) y^2 + q_{\lambda, \mu}(t),
\end{array}
\right.
\end{equation}
noticing that the Neumann boundary conditions translate here into $y(0) = y(T) = 0$. Then, on varying of the parameters $\lambda, \mu$, we look for intersections between the two planar curves obtained by shooting the $x$-axis forward (from $0$ to $\tau$) and backward (from $T$ to $\tau$) in time.
A very similar argument was recently exploited in \cite{BosGar16}, dealing with nonlinearities on the whole real line, 
the main difference being here the possible appearance of a blow-up phenomenon for the solutions to \eqref{sistema}. 
The drawback of our approach is that it may require involved computations when the weight $a(t)$ is not a two-step function (this being the reason why we limit ourselves to the configuration \eqref{ipotesia}); on the other hand, we stress that the case $\mu < \mu_0(\lambda)$, which actually motivates our investigation, 
seems quite delicate to be tackled using functional analytic techniques (see Remark \ref{variazione}). 
\smallbreak
The detailed proof of Theorem \ref{thmain} will be given in Section \ref{sezdue}. The final Section \ref{seztre}
is instead devoted to some further comments, variants and corollaries of our result.

\section{Proof of Theorem \ref{thmain}}\label{sezdue}

\subsection{A topological lemma}

We first state a topological lemma concerning planar curves disconnecting $\mathbb{R}^2$,
which can be seen as a variant of the Jordan Curve Theorem. We believe that such a result is well known; however,
since it seems not easy to find an appropriate reference, we provide an explicit proof.

\begin{lemma}\label{topologico}
Let $\gamma: \mathbb{R} \to \mathbb{R}^2$ be a continuous and injective function
satisfying
\begin{equation}\label{propria}
\lim_{\vert t \vert \to +\infty} \vert \gamma(t) \vert = +\infty.
\end{equation}
Then
$\mathbb{R}^2 \setminus \gamma(\mathbb{R})$ consists exactly of two connected components, both unbounded.
\end{lemma}

\begin{proof}
Condition \eqref{propria} implies that $\gamma$ can be extended in a continuous way to a map $\tilde{\gamma}: \mathbb{S}^1 \to \mathbb{S}^2$, where $\mathbb{S}^d$, $d \geq 1$, is the one-point compactification of $\mathbb{R}^d$, i.e., $\mathbb{S}^d = \mathbb{R}^d \cup \{\infty\}$, with canonical injection $j_d: \mathbb{R}^d \to \mathbb{S}^d$. Then, $\tilde{\gamma}$ is a Jordan curve on the sphere $\mathbb{S}^2$, so that, by the Jordan Curve Theorem on spheres, $\mathbb{S}^2 \setminus \tilde\gamma(\mathbb{S}^1)=\mathcal{C}_1 \sqcup \mathcal{C}_2$, where $\mathcal{C}_i$, $i=1, 2$ are connected clopen sets in $\mathbb{S}^2 \setminus \tilde\gamma(\mathbb{S}^1)$. Since $\{\infty\} \notin \mathcal{C}_i$, $i=1, 2$, we have
$$
\mathbb{R}^2 \setminus \gamma(\mathbb{R}) = j_2^{-1} (\mathbb{S}^2 \setminus \tilde\gamma(\mathbb{S}^1)) = j_2^{-1} (\mathcal{C}_1 \sqcup \mathcal{C}_2) =j_2^{-1} (\mathcal{C}_1) \sqcup j_2^{-1}(\mathcal{C}_2),
$$  
where both the sets $j_2^{-1} (\mathcal{C}_1)$ and $j_2^{-1}(\mathcal{C}_2)$ are connected and clopen in $\mathbb{R}^2 \setminus \gamma(\mathbb{R})$. The unboundedness of such connected components follows from the fact that $\{\infty\}$ is a boundary point both for $\mathcal{C}_1$ and $\mathcal{C}_2$. 
\end{proof}

Incidentally, we notice that condition \eqref{propria} holds true if and only if $\gamma$ is a proper map, i.e,
preimages of compact sets are compact.

\subsection{The shooting argument}

We first adapt to our setting the change of variables introduced in \cite{BosZan15}.
Set
$$
W(u)=-\int_1^u \frac{dx}{g(x)}, \quad \mbox{ for every } u > 0;
$$
in view of \eqref{fondamentale}, \eqref{ipoinf} and \eqref{ipozero}, $W(u)$ is a strictly decreasing $C^2$-diffeomorphism of 
$\,]0,+\infty[\,$ onto $\mathbb{R}$ (actually, $g'(x) \to 0$ both for $x \to 0^+$ and for
$x \to +\infty$ implies that the two integrals
$\int_0^1 1/g$ and $\int_1^{+\infty}1/g$ are divergent).
Then, $u(t)$ is a \emph{positive} solution to the differential equation in \eqref{eqintro2}
if and only if
\begin{equation}\label{ilcambio}
x(t):=W(u(t))
\end{equation}
solves  
\begin{equation}\label{scalare}
x'' = h(x) (x')^2 + q_{\lambda,\mu}(t),
\end{equation}
where $h: \mathbb{R} \to \mathbb{R}$ is the continuous function given by 
$$
h(x)=g'(W^{-1}(x)).
$$
As a consequence of \eqref{ipoinf} and \eqref{ipozero}, we have
\begin{equation}\label{ipotesih}
h(x)x > 0, \quad \mbox{ for every } \vert x \vert > d := \max\{\vert W(r) \vert, \vert W(R) \vert\},
\end{equation}
and
\begin{equation}\label{ipotesih2}
\lim_{\vert x \vert \to +\infty} h(x) = 0.
\end{equation}
Moreover, since $x'(t) = -u'(t)/g(u(t))$, we have that $u(t)$ satisfies the Neumann boundary conditions in \eqref{eqintro2} if and only if
$x(t)$ does. 
\smallbreak
Now, the situation may look like the one considered in \cite[Theorem 2.2]{BosGar16}. However, here we have to take into account the possible failure
of the global continuability of the solutions to \eqref{scalare}; notice, indeed, that the equivalence between 
$u'' + q_{\lambda,\mu}(t)g(u) = 0$ and \eqref{scalare} is guaranteed only as long as $u(t) > 0$. While, on one hand, this requires a refinement of the techniques therein, on the other hand it gives rise to a richer picture of solvability (that is, we will prove that $\mu^+(\lambda) = +\infty$ for $\lambda$ large).
\smallbreak
We then proceed with the proof.
Let us start by writing \eqref{scalare} as the equivalent planar system
\begin{equation}\label{sistema3}
\left\{
\begin{array}{ll}
x' = y \vspace{0.2 cm} \\
y' = h(x) y^2 + q_{\lambda,\mu}(t) \\
\end{array} \right.
\end{equation}
and, for every $z_0 = (x_0,y_0) \in \mathbb{R}^2$ and $t_0\in [0,T]$,
denote by 
$$
\zeta_{\lambda,\mu}(\cdot;t_0,z_0) = (x_{\lambda,\mu}(\cdot;t_0,z_0),y_{\lambda,\mu}(\cdot;t_0,z_0))
$$
the solution to \eqref{sistema3} with $\zeta_{\lambda,\mu}(t_0;t_0,z_0) = z_0$, whenever this is defined.
The standard theory of ODEs guarantees that such a map is continuous (in all its variables, including the parameters $\lambda,\mu$) on its domain.
\smallbreak
We first focus on the behaviour of the solutions in the interval $[\tau,T]$, thus considering the map
$\zeta_\mu(t;T;z_0) = \zeta_{\lambda,\mu}(t;T,z_0)$ (indeed, in this time interval the parameter $\lambda$ does not appear). 

\begin{lemma}\label{ramo-}
For every $\mu > 0$ and $x_0 \in \mathbb{R}$, the backward solution $\zeta_\mu(\cdot;T,(x_0,0))$ is defined on the whole interval
$[\tau,T]$. Moreover, for any $\epsilon \in \,]0,1]$, and $\mu_1, \mu_2$ with  $0 < \mu_1 < \mu_2$, there exists $N_\epsilon > 0$ such that, for every $\vert x_0 \vert \geq N_\epsilon$ and $\mu \in [\mu_1, \mu_2]$,
it holds
\begin{equation}\label{ramo-f}
\left \vert y_\mu(\tau;T,(x_0,0)) - \mu \int_\tau^T a^-(t)\,dt \right \vert < \epsilon, \qquad
\vert x_\mu(\tau;T,(x_0,0))  - x_0 \vert < N,
\end{equation}
where $N$ is a suitable constant depending on $a^-(t), \mu_1$ and $\mu_2$. 
\end{lemma}

\begin{proof}
To prove that the solution $\zeta_\mu(\cdot;T,(x_0,0))$ is defined on the whole interval
$[\tau,T]$, it is enough to observe that any (backward) solution $u(t)$ to $u'' - \mu a^-(t)g(u) = 0$ satisfying 
$u(T) > 0$ and $u'(T) = 0$ is convex and decreasing on any subinterval $[t^*,T] \subset [\tau,T]$ where it exists.
Hence, $u(t) \geq u(T) > 0$ avoids blow-up phenomena on the boundary $\{ u = 0 \}$;
on the other hand, $u(t)$ cannot blow-up at infinity since $g(u)$ is bounded. 

The second part of the statement follows from \cite[Lemma 5]{BosZan15}. 
\end{proof}

Next, we consider the solutions in the interval $[0,\tau]$, dealing with the map
$\zeta_\lambda(t;0;z_0) = \zeta_{\lambda,\mu}(t;0,z_0)$ (now, symmetrically, the parameter $\mu$ does not appear). 
Here the situation changes a bit.

\begin{lemma}\label{ramo+}
Let us fix $\lambda > 0$. Then, two (mutually excluding) alternatives can occur:
\begin{itemize}
\item[(A)] for any $x_0 \in \mathbb{R}$, the solution $\zeta_\lambda(\cdot;0,(x_0,0))$ is defined on the whole interval
$[0,\tau]$;
\item[(B)] there exist $x_*,x^* \in \mathbb{R}$, with $x_* \leq x^*$, such that the solution $\zeta_\lambda(\cdot;0,(x_0,0))$ is defined on the whole interval
$[0,\tau]$ for $x \notin [x_*,x^*]$ and
\begin{equation}\label{divergenza}
\lim_{x_0 \to (x_*)^-} y_\lambda(\tau;0,(x_0,0)) = \lim_{x_0 \to (x^*)^+} y_\lambda(\tau;0,(x_0,0)) = +\infty.
\end{equation}
\end{itemize}
In both cases, for any $\epsilon \in \,]0,1]$, there exists $M_\epsilon > \max\{\vert x_* \vert, \vert x^* \vert\}$ such that, for every $\vert x_0 \vert \geq M_\epsilon$ it holds
\begin{equation}\label{ramo+f}
\left \vert y_\lambda(\tau;0,(x_0,0)) - \lambda \int_0^\tau a^+(t)\,dt \right \vert < \epsilon, \qquad
\vert x_\lambda(\tau;0,(x_0,0))  - x_0 \vert < M,
\end{equation}
where $M$ is a suitable constant depending on $a^+(t)$ and $\lambda$. 
\end{lemma}

\begin{proof}
First, again from \cite[Lemma 5]{BosZan15} we have that there exists $M > 0$ such that 
the solution $\zeta_\lambda(\cdot;0,(x_0,0))$ exists on the whole $[0,\tau]$ when $\vert x_0 \vert \geq M$; moreover, \eqref{ramo+f} holds true.
We now define
$$
x_* = \sup\{x \in \mathbb{R}  : \,\zeta_\lambda(\cdot;0,(x_0,0)) \, \mbox{ is defined on } [0,\tau] \mbox{ for every } x_0 < x\}
$$
and
$$
x^* = \inf\{x \in \mathbb{R}  : \,\zeta_\lambda(\cdot;0,(x_0,0)) \, \mbox{ is defined on } [0,\tau] \mbox{ for every } x_0 > x\};
$$
notice that both the sets are non-empty in view of the previous discussion, so that
$-\infty < x_*$ and $x^* < +\infty$. It is clear that 
$$
x_* > x^* \; \Longleftrightarrow \; x_* = +\infty \; \Longleftrightarrow \; x^* = -\infty
$$
and this holds if and only if case (A) of the statement occurs. Thus, we assume that this does not happen
and we prove \eqref{divergenza}.

Assume by contradiction that there exists a sequence $x_0^n \nearrow x_*$ such that
$y_n(t):= y_\lambda(t;0,(x_0^n,0))$ satisfies $y_n(\tau) \leq \Theta$ for any $n$ and a suitable $\Theta > 0$. 
Going back to the original variables, this means that the solution $u_n(t)$ to the Cauchy problem 
$$
\left\{
\begin{array}{l}
u_n'' + \lambda a^+(t) g(u_n) = 0  \vspace{0.2 cm}\\
u_n(0) = u_0^n := W^{-1}(x_0^n)  \vspace{0.2 cm}\\ u'_n(0) = 0 
\end{array}
\right.
$$
fulfills
\begin{equation}\label{assurdo}
-\frac{u_n'(\tau)}{g(u_n(\tau))} \leq \Theta.
\end{equation}
We claim that $u_n(\tau) \to 0$. Indeed, if this is not the case, since $u_n(t)$
is concave and decreasing there exists $\delta > 0$ such that $\delta \leq u_n(\tau) \leq u_n(t) \leq u_0^n$
for every $t \in [0,\tau]$. Then, a standard compactness argument yields a positive solution
to $u'' + \lambda a^+(t)g(u) = 0$, defined on the whole $[0,\tau]$ and satisfying $u(0) = W^{-1}(x_*)$, $u'(0) = 0$.
This contradicts the definition
of $x_*$ and the claim is proved. In view of \eqref{assurdo}, we infer that $u'_n(\tau) \to 0$. Again in view of the concavity of $u_n(t)$,
$0 \geq u_n'(t) \geq u_n'(\tau) \to 0$ for every $t \in [0,\tau]$, whence $u_n' \to 0$ uniformly on $[0,\tau]$.
Since $u_n(\tau) \to 0$, it follows that $u_n \to 0$ uniformly on $[0,\tau]$, against the fact that $u_0^n \to W^{-1}(x_*) \neq 0$. 
The case $x_0^n \searrow x^*$ is completely analogous.
\end{proof}

We now continue the proof taking first into account case (A) of Lemma \ref{ramo+}. In this setting, the argument is really the same as the one in \cite[Theorem 2.2]{BosGar16}, after having observed that the set
$$
\Gamma^+_\lambda = \{\zeta_\lambda(\tau;0, (x_0, 0))\, : \, x_0 \in \mathbb{R}\} \subset \mathbb{R}^2
$$
disconnects the plane into two connected components. In the framework therein, this was a consequence of the global continuability of the solutions
(implying that the Poincar\'e map is a global homeomorphism of the plane onto itself), while now this comes from Lemmas \ref{ramo+} and \ref{topologico}.
Indeed, from the second inequality in \eqref{ramo+f} we have that the parametrized curve $\mathbb{R} \ni x_0 \mapsto \zeta_\lambda(\tau;0,(x_0,0))$ satisfies
condition \eqref{propria}.
\smallbreak
However, for future convenience, we give an informal flavour of the above argument, always referring to \cite[Section 3]{BosGar16} for the technical details.
We have
$$
\mathbb{R}^2 \setminus \Gamma^+_\lambda = \mathcal{O}_d \cup \mathcal{O}_u,
$$
where $\mathcal{O}_d,\mathcal{O}_u$ are connected clopen sets in $\mathbb{R}^2 \setminus \Gamma^+_\lambda$; precisely, we choose $\mathcal{O}_d$
(resp., $\mathcal{O}_u$) as the component ``below'' (resp., ``above'') $\Gamma^+_\lambda$. We now set
$$
\Gamma^-_{\mu} = \{\zeta_{\mu}(T;\tau, (x_0, 0))\, : \, x_0 \in \mathbb{R}\} \subset \mathbb{R}^2
$$
(which is well-defined in view of Lemma \ref{ramo-}) and we observe that points in the intersection
$
\Gamma^+_\lambda \cap \Gamma^-_\mu
$
are in a one-to-one correspondence with Neumann solutions to \eqref{scalare}.
Thus we search for such intersections, focusing at first on the mutual position between $\Gamma^+_\lambda$ and $\Gamma^-_{\mu_0(\lambda)}$.
The first inequalities in \eqref{ramo-f} and \eqref{ramo+f} ensure that the $y$-components of $\Gamma^+_\lambda$ and
$\Gamma^-_{\mu_0(\lambda)}$ approach the same value
$$
\lambda \int_0^\tau a^+(t)\,dt = \mu_0(\lambda) \int_\tau^T a^-(t)\,dt.
$$
However, more can be said (cf. \cite[Lemma 3.5]{BosGar16}): precisely, the right tail of $\Gamma^-_{\mu_0(\lambda)}$ (that is, the set
$\{\zeta_{\mu_0(\lambda)}(T;\tau, (x_0, 0)) \, : x_0 \gg 0\}$) lies in $\mathcal{O}_d$, while the left tail
(defined analogously for $x_0 \ll 0$) lies in $\mathcal{O}_u.\,$\footnote{This is the key point of the argument and it is a consequence of the sign condition assumed on $g'(x)$. Indeed, were the right tail in $\mathcal{O}_u$, by slightly decreasing $\mu$ it would intersect $\Gamma^+_{\lambda}$ 
(use again the first inequalities in \eqref{ramo-f}, \eqref{ramo+f}); this would be a large Neumann solution
to \eqref{scalare}, which is prohibited in view of \eqref{ipotesih} (cf. \cite[Proposition 3.1]{BosGar16}, that is, just integrate the equation).}
This ensures the existence of at least one intersection between $\Gamma^+_{\lambda}$ and 
$\Gamma^-_{\mu_0(\lambda)}$; clearly enough, such an intersection persists when slightly varying $\mu$ in a (two-sided) neighborhood of
$\mu_0(\lambda)$. On the other hand, for $\mu$ in a sufficiently small neighborhood of $\mu_0(\lambda)$ a second intersection 
between $\Gamma^+_\lambda$ and $\Gamma^-_\mu$ comes from the tails: precisely, for $\mu < \mu_0(\lambda)$ it appears   
on the left-tail, while for $\mu > \mu_0(\lambda)$ it appears on the right one (of course, we are using once more the inequalities in \eqref{ramo-f}, \eqref{ramo+f}).
\smallbreak
We now deal with case (B) in Lemma \ref{ramo+}. In this situation, we can find $\delta > 0$ so small that
$$
y_\lambda(\tau; 0, (x_0, 0)) > 2\, \max_{x_0 \in \mathbb{R}} \,  y_{\mu_0(\lambda)}(T; \tau, (x_0, 0)), \quad \textrm{ for  } x_0 \in \, ]x_*-\delta, x_*[\, \cup \,]x^*, x^*+\delta[
$$
and we define  
$$
\Gamma^+_\lambda = \{\zeta_\lambda(\tau;0, (x_0, 0))\, : \, x_0 \in \mathbb{R} \setminus [x_*-\delta, x^*+\delta]\} \cup \mathcal{S}_\lambda,
$$
where $\mathcal{S}_\lambda$ is the segment joining the points $\zeta_\lambda(\tau; 0, (x_*-\delta, 0))$ and $\zeta_\lambda(\tau; 0, (x^*+\delta, 0))$.
Clearly, this set still disconnects $\mathbb{R}^2$, in view of Lemmas \ref{ramo+} and \ref{topologico}; moreover, its
tails have the same properties as before. We can thus check that the very same arguments apply, yielding one intersection
between $\Gamma^+_\lambda$ and $\Gamma_{\mu_0(\lambda)}$ and two intersections
between $\Gamma^+_\lambda$ and $\Gamma^-_\mu$ for $\mu$ in a small deleted neighborhood of $\mu_0(\lambda)$.
Up to shrinking such a neighborhood of $\mu_0(\lambda)$ we can assume that,
for any $\mu$ therein,
$$
\max_{x_0 \in \mathbb{R}} \,  y_{\mu}(T; \tau, (x_0, 0)) < 2 \max_{x_0 \in \mathbb{R}} \,  y_{\mu_0(\lambda)}(T; \tau, (x_0, 0)),
$$
so that $\mathcal{S}_\lambda \cap \Gamma^-_\mu = \emptyset$ and the produced intersections still give rise to Neumann solutions
to \eqref{scalare}. 
\smallbreak
We have thus proved the first part of Theorem \ref{thmain}; to conclude the proof, we still have to show that $\mu^+(\lambda) = +\infty$ if $\lambda \gg 0$, that is, for $\lambda$ large there are always two solutions for any $\mu > \mu_0(\lambda)$. 

At first, we claim that there exists $\lambda^* > 0$, depending only on $g(u)$ and $a^+(t)$, such that for $\lambda > \lambda^*$ we are surely in case (B)
of Lemma \ref{ramo+}. Actually, we are going to show that if $u(t)$ is a positive solution of $u'' + \lambda a^+(t)g(u) = 0$
satisfying $u(0) = 1$, $u'(0) = 0$ and defined on the whole $[0,\tau]$, then $\lambda$ must be smaller than a constant (not depending on $u(t)$).
As a first step, we fix
$\epsilon > 0$ so small that $a_\epsilon:=\int_0^{\tau-\epsilon}a^+(t)\,dt > 0$; by integrating the equation on $[0,\tau-\epsilon]$, we find
\begin{equation}\label{stl}
\lambda =  -\frac{u'(\tau-\epsilon)}{\int_0^{\tau-\epsilon}a^+(t)g(u(t))\,dt}.
\end{equation}
Now, noticing that $u(t)$ must be concave on $[0,\tau]$, we have on one hand
$$
1 \geq u(\tau)-u(\tau-\epsilon) = -\int_{\tau}^{\tau-\epsilon}u'(s)\,dt \geq -u'(\tau-\epsilon)\epsilon.
$$
On the other hand, again by a concavity argument, $u(t) \geq \tfrac{1}{\tau}(\tau-t)$ for any $t \in [0,\tau]$ so that
$$
u(t) \geq \frac{\epsilon}{\tau}, \quad \mbox{ for any } t \in [0,\tau-\epsilon];
$$
hence
$$
g(u(t)) \geq m_\epsilon:= \min_{u \geq \epsilon/\tau}g(u), \quad \mbox{ for any } t \in [0,\tau-\epsilon].
$$
Going back to \eqref{stl}, it follows that
$$
\lambda \leq \frac{1}{\epsilon m_\epsilon a_\epsilon} =: \lambda^*,
$$
as desired.

We now fix $\lambda > \lambda^*$ and, using the notation in case (B) of Lemma \ref{ramo+}, we set  
$$
\Gamma^{+, l}_\lambda= \{\zeta_\lambda(\tau; 0, (x_0, 0)) \, : \, x_0 < x_*\}
$$
and 
$$
\Gamma^{+, r}_\lambda= \{\zeta_\lambda(\tau; 0, (x_0, 0)) \, : \, x_0 > x^*\},
$$ 
observing that $\Gamma^{+, l}_\lambda \cap \Gamma^{+, r}_\lambda = \emptyset$. 
Fixed $\mu > \mu_0(\lambda)$,
this time we see that the set $\Gamma^-_\mu$ disconnects the plane
(in view of Lemmas \ref{ramo-} and \ref{topologico}). Similarly as before, we can then write
$$
\mathbb{R}^2 \setminus \Gamma^-_\mu = \mathcal{A}_d \cup \mathcal{A}_u,
$$
where $\mathcal{A}_d,\mathcal{A}_u$ are connected clopen sets in $\mathbb{R}^2 \setminus \Gamma^-_\mu$. Arguing as in \cite[Lemma 3.4]{BosGar16}, we can prove that such sets have the following property: 
\smallbreak
\noindent
$\bullet$ \emph{for every $\bar{y} \neq \mu\int_\tau^T a^-(t)\,dt$, there exists
$\eta(\bar{y}) > 0$ such that}
$$
\big(\,]-\infty,-\eta(\bar{y})] \times [\bar{y},+\infty[\,\big) \cup
\big([\eta(\bar{y}),+\infty[\, \times [\bar{y},+\infty[\,\big)
 \subset \mathcal{A}_u \quad \mbox{ if } \; \bar{y} > \mu\int_\tau^T a^-(t)\,dt
$$
\emph{and}
$$
\big(\,]-\infty,-\eta(\bar{y})] \times \,]-\infty,\bar{y}]\big) \cup
\big([\eta(\bar{y}),+\infty[\, \times \,]-\infty,\bar{y}]\big)
\subset \mathcal{A}_d \quad \mbox{ if } \; \bar{y} < \mu\int_\tau^T a^-(t)\,dt.
$$
In particular, since the $y$-coordinate on the set $\Gamma^-_\mu$ is bounded, this implies that the connected component $\mathcal{A}_d$ (resp. $\mathcal{A}_u$) contains all the horizontal lines $y=\bar{y}$ with $\bar{y} \ll 0$ (resp. $\bar{y} \gg 0$). 
Hence, Lemma \ref{ramo-} implies that both $\Gamma^{+, l}_\lambda$ and $\Gamma^{+, r}_\lambda$ own points in both the connected components $\mathcal{A}_d, \mathcal{A}_u$, so that, by an elementary connectedness argument, we deduce 
$$
\Gamma^{+, l}_\lambda \cap \Gamma^-_\mu \neq \emptyset \neq \Gamma^{+, r}_\lambda \cap \Gamma^-_\mu,
$$
thus finding the desired two solutions.

\begin{figure}[!h]
\centering
\includegraphics[scale=0.8]{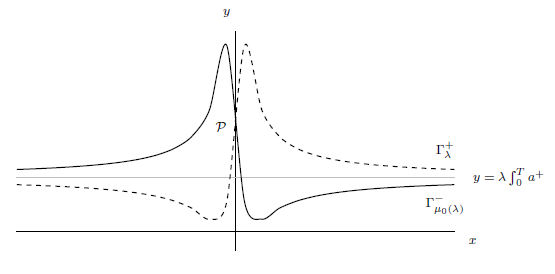}
\label{globale}
\begin{flushleft}
\caption{\footnotesize We depict the mutual position of the curves $\Gamma_\lambda^+$ (dashed) and $\Gamma_{\mu_0(\lambda)}^-$ in
the
setting of alternative (A) in Lemma \ref{ramo+}; namely, the curve $\Gamma_\lambda^+$ is
global (as observed in Remark \ref{Nonesiste}, this is always the case when $\lambda>0$ is small enough).
The $y$-components of the two curves approach the same value $\lambda \int_0^T a^+ = \mu_0(\lambda) \int_0^T a^-$ in view of Lemmas
\ref{ramo-} and \ref{ramo+}, with mutual position as described throughout the proof. We notice the presence of an intersection
$\mathcal{P}$ between the two curves; clearly enough, when moving $\mu$ in a sufficiently small neighborhood of
$\mu_0(\lambda)$, such an intersection persists and a new one comes from the tails (on the left for $\mu \in \,]\mu_-(\lambda),
\mu_0(\lambda)[\,$, on the right for $\mu \in \,]\mu_0(\lambda), \mu_+(\lambda)[\,$). It is also possible to see that, taking $\mu$ too large
or too small, the two curves become disjoint (agreeing with Remark \ref{Nonesiste}).
}
\end{flushleft}
\end{figure}

\begin{figure}[!h]
\centering
\includegraphics[scale=0.8]{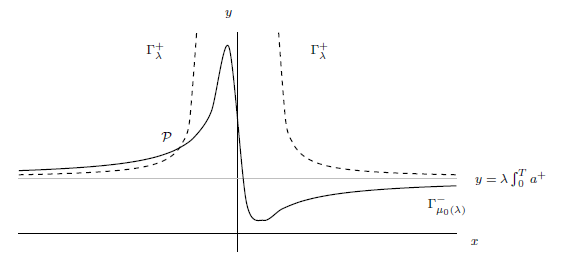}
\label{esplode}
\begin{flushleft}
\caption{\footnotesize We depict the mutual position of the curves $\Gamma_\lambda^+$ (dashed) and $\Gamma_{\mu_0(\lambda)}^-$ in
the
setting of alternative (B) in Lemma \ref{ramo+}; here, the curve $\Gamma_\lambda^+$ ``breaks'' into two
branches (this being always the case for $\lambda > \lambda^*$).
Again, the $y$-components of the two curves approach the same value and a first intersection $\mathcal{P}$ between
$\Gamma_\lambda^+$ and $\Gamma_{\mu_0(\lambda)}^-$ comes between the mutual position of the tails. When moving
$\mu$ in a \emph{left} neighborhood of $\mu_0(\lambda)$, the behavior is the same as in the previous case (with
a second intersection appearing on the left for $\mu$ not too small); here the novelty appears for $\mu > \mu_0(\lambda)$,
since both $\mathcal{P}$ and the intersection coming from the right tail persist for \emph{every}
$\mu > \mu_0(\lambda)$.}
\end{flushleft}
\end{figure}

\section{Some miscellaneous remarks}\label{seztre}

\begin{remark}\label{Nonesiste}
\textbf{Nonexistence issues.} We remark that, arguing as in \cite[Theorem 4.1]{BosGar16}, it is possible to show that, for any fixed $\lambda > 0$, \eqref{eqintro2} has no positive solutions for $\mu > 0$ small enough. This implies that the condition $\mu_-(\lambda) > 0$ in Theorem \ref{thmain} cannot be improved into $\mu_-(\lambda) = 0$, explaining the expression ``positive but small enough'', used throughout the paper referring to the average of the weight function $q_{\lambda,\mu}(t) = \lambda a^+(t) - \mu a^-(t)$. 

On the other hand, wishing to briefly comment on the finiteness of $\mu^+(\lambda)$ (when $\lambda > 0$ is small),
we first notice that, by using Sturm comparison techniques (at zero and at infinity) together with continuous dependence arguments, it is not difficult to see
that, when $\lambda > 0$ is small enough, we are in case (A) of Lemma \ref{ramo+}. Once this is established, one can argue again as in \cite[Theorem 4.1]{BosGar16} to ensure that \eqref{eqintro2} has no positive solutions for $\mu$ large enough (up to a slightly stronger condition on 
$a(t)$). This means that $\mu_+(\lambda) < +\infty$ for $\lambda$ small, clarifying the meaning of the expression
``negative but small enough'' used in the Introduction.

See also Figure \ref{globale} and its caption for a graphical explanation.  
\end{remark}

\begin{remark}\label{variazione}
\textbf{A functional-analytic viewpoint.}
In order to highlight the peculiarity of the case $\mu \leq \mu_0(\lambda)$ treated in our paper, 
we think that a comparison with the functional analytic approach to problem \eqref{eqintro2}
can be helpful. In particular, we first focus on the variational approach, following \cite{BosZan12}.
After having extended $g(u)$ to the whole real line line as an odd function and setting
$G(u) = \int_0^u g(x)\,dx$, positive solutions to \eqref{eqintro2} can be found as critical points of the action functional
$$
J_{\lambda,\mu}(u) = \frac{1}{2}\int_0^T u'(t)^2 \,dt - \int_0^T F_{\lambda,\mu}(t,u(t))\,dt, \qquad u \in H^1(0,T), 
$$
satisfying $u(t) > 0$ for every $t \in [0,T]$, where $F_{\lambda,\mu}(t,u) = q_{\lambda,\mu}(t)G(u)$.
In this setting, the condition $\mu > \mu_0(\lambda)$ plays a crucial role in two different directions: 
on one hand (together with a mild technical superlinearity assumption for $g(u)$ at zero) it ensures that the origin 
$u \equiv 0$ is a strict local minimum for $J_{\lambda,\mu}$; on the other hand 
(again together with a suitable sublinearity assumption at infinity), it ensures the coercivity of $J_{\lambda,\mu}$
(compare with Ahmad-Lazer-Paul type conditions in \cite[Theorem 4.8]{MawWil89}). Then, two critical points can be provided, via global minimization and the Mountain
Pass lemma, whenever there exists $\bar{u} \in H^1(0,T)$ such that $J_{\lambda,\mu}(\bar{u}) < 0$ (see Remark \ref{parametri} below for more comments
on this point); it is standard to see that both these critical points correspond to \emph{positive} solutions to \eqref{eqintro2}. For $\mu \leq \mu_0(\lambda)$, it is clear that all this discussion fails, being the geometry of the functional of completely different and not easily detectable nature. 

Similar difficulties would arise wishing to use a topological approach as in \cite{BosFelZan16}. Without going into the details, 
the assumption $\mu > \mu_0(\lambda)$ is crucial in order to show that the coincidence degree of a suitable operator associated with \eqref{eqintro2} equals $1$ on small and large balls (being on the other hand equal to zero on intermediate ones and thus providing the desired pair of solutions),
while an analogous argument seems to be difficult in the case $\mu \leq \mu_0(\lambda)$. 
\end{remark}

\begin{remark}\label{parametri}
\textbf{The role of $\lambda$.}
The role of the parameter $\lambda$ deserves some comments, as well. Indeed, when $\lambda$ is large enough one can easily find $\bar{u} \in H^1(0, T)$ such that $J_{\lambda, \mu}(\bar{u}) < 0$ independently of $\mu$ (just  by taking $\textnormal{supp}(\bar{u})$ contained in an interval where $a(t) > 0$), thus providing a pair of positive solutions to \eqref{eqintro2} for any $\mu > \mu_0(\lambda)$, according to the discussion in Remark \ref{variazione} above.
This agrees with the recent contributions \cite{BosFelZan16,BosZan12}.
From this point of view, the possibility of finding positive solutions for \emph{any} positive $\lambda$
(but $\mu$ in a bounded neighborhood of $\mu_0(\lambda)$ when $\lambda$ is small) could then 
seem a quite relevant novelty of Theorem \ref{thmain}. However,
on one hand this is likely to be provable also with variational techniques, for $\mu > \mu_0(\lambda)$
(carefully evaluating the functional $J_{\lambda,\mu_0(\lambda)}$ on ``small, almost constant'' functions, so as to prove $\inf J_{\lambda,\mu} < 0$ for $\mu$ near $\mu_0(\lambda)$).
On the other hand, we stress that the solvability picture given by Theorem \ref{thmain} has not to be misunderstood with the one for 
the single parameter equation 
\begin{equation}\label{naturale}
u''+\lambda a(t) g(u)=0
\end{equation}
(that is, the differential equation in \eqref{eqintro2} for $\mu=\lambda$), which is indeed the natural one when studying super-sublinear problems. Actually, 
positive Neumann solutions to \eqref{naturale} do not exist for $\int_0^T a(t)\,dt < 0$ and $\lambda$ small \cite[Theorem 1.2]{BosFelZan16}. This, however,
is not a contradiction. Indeed, the first part of Theorem \ref{thmain} is of local nature with respect to $\mu$, providing solutions only in a neighborhood of
$\mu_0(\lambda)$; from this point of view, we could equivalently have set $\lambda=1$ and dealt with the equation
$u'' + (a^+(t) - \hat{\mu} a^-(t)) g(u) = 0$. Anyway, it is meaningful to keep the parameter $\lambda$ as well, since this leads to the second part of our statement (namely, $\mu_+(\lambda) = +\infty$ for $\lambda$ large) which recovers part of the previously mentioned results. See Figure \ref{pianolambdamu} below.
\end{remark}

\begin{remark}
\textbf{Increasing nonlinearities.}
By combining the arguments in this paper with the ones in \cite{BosGar16}, it is possible to prove the following:
\smallbreak
\noindent
\textit{Assume that \eqref{fondamentale} and \eqref{ipozero} hold true and that
$$
g'(u) > 0 \; \textnormal{ for every } u > R \quad \textnormal { and } \quad \lim_{u \to +\infty} g'u)=0; 
$$
let \eqref{ipotesia} hold, as well. Then, for any $\lambda > 0$ there exists $\mu_+(\lambda) > \mu_0(\lambda)$
such that problem \eqref{eqintro2} has at least two positive solutions for any $\mu \in \,]\mu_0(\lambda), \mu_+(\lambda)[\,$. 
Moreover, $\mu_+(\lambda) = +\infty$
for any $\lambda$ sufficiently large, that is, for $\lambda > \lambda^*$, with $\lambda^* > 0$ depending only on $g(u)$ and
$a^+(t)$.}
\smallbreak
\noindent
This result complements previously known statements by taking into account the case $\lambda > 0$ small, according to the discussion
in Remark \ref{parametri}. Of course, here we cannot in general expect existence for $\mu \leq \mu_0(\lambda)$, since
nonlinearities with $g'(u) > 0$ for any $u > 0$ are allowed (recall \eqref{med}).
\end{remark}

\begin{remark}
\textbf{Periodic and radially symmetric solutions.}
We finally recall that Theorem \ref{thmain} can be used to
construct positive $T$-periodic solutions to the differential equation
in \eqref{eqintro2} when $a(t)$ is $T$-periodic, satisfies $a(\sigma + t)=a(\sigma - t)$ for some $\sigma \in [0, T[\,$ and almost every $t \in \mathbb{R}$ and, for a suitable $\tau' \in \,]0, T/2[\,$,
$$
a(t) \geq 0 \; \mbox{ for a.e. } t \in [\sigma, \sigma + \tau'], \qquad 
a(t) \leq 0 \; \mbox{ for a.e. } t \in [\sigma+\tau', \sigma + T/2].
$$
Indeed, after solving the Neumann BVP on $[\sigma, \sigma+T/2]$ 
a positive $T$-periodic solution can be obtained by a symmetry extension with respect to $\sigma$ (cf. \cite[Corollary 4]{BosZan15} for further details). 

On the other hand, Theorem \ref{thmain} also provides positive radial solutions to the Neumann BVP associated with an elliptic equation like
\begin{equation}\label{anello}
\Delta u + \big( \lambda a^+(x) - \mu a^-(x)\big) g(u) = 0, \qquad x \in \mathcal{O},
\end{equation}
where $\mathcal{O} = \{ x \in \mathbb{R}^d : 0 < r_1 < \vert x \vert < r_2 \}$ is an open annulus around the origin and the weight
$a(x) = a^+(x)-a^-(x)$ is radially symmetric, that is, $a(x) = \mathcal{A}(\vert x \vert)$.
Indeed, looking for a positive radial solution $\mathcal{U}(\vert x \vert)$ to \eqref{anello} yields the ODE 
\begin{equation}\label{radiale}
\mathcal{U}'' + \frac{d-1}{r}\;\mathcal{U}' + \big( \lambda \mathcal{A}^+(r) - \mu \mathcal{A}^-(r)\big)g(\mathcal{U}) = 0, \qquad r = \vert x \vert \in [r_1,r_2],
\end{equation}
and the original Neumann boundary conditions read as $\mathcal{U}'(r_1) = \mathcal{U}'(r_2) = 0$. Then, a standard change of variables transforms
\eqref{radiale} into an equation of the type $v'' + (\lambda b^+(t) - \mu b^-(t))g(v) = 0$,
where $b(t)$ is an indefinite weight with the same shape of $a(\vert x \vert)$. For more details, see
\cite[Section 5.1]{BosFelZan16}.
\end{remark}

\begin{figure}[!h]
\centering
\includegraphics[scale=0.8]{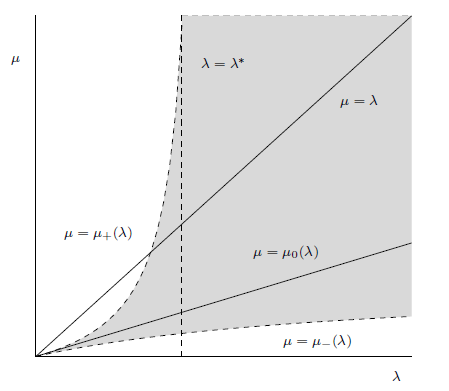}
\begin{flushleft}
\caption{\footnotesize We give a pictorial description of Theorem \ref{thmain}, showing the solvability regions for
positive solutions to \eqref{eqintro2} in the $(\lambda,\mu)$-plane. For a better comparison with the 
existing literature, we choose a weight function $a(t)$ with negative mean value, that is, $\int_0^T a(t)\,dt < 0$,
implying that the critical line $\mu = \mu_0(\lambda)$ lies below the bisecting line $\mu = \lambda$.
We see that solvability is ensured in the shaded region bounded by the curves $\mu = \mu_-(\lambda)$ and
$\mu = \mu_+(\lambda)$, lying respectively below and above the line $\mu = \mu_0(\lambda)$; in particular, 
$\mu_+(\lambda) = +\infty$ for $\lambda > \lambda^*$, implying that the whole epigraph of $\mu_-(\lambda)$ is 
contained in the solvability region. We observe that this picture is compatible with the known non-existence results
for the equation $u'' + \lambda a(t)g(u) = 0$ when $\lambda$ is small; indeed, the curve $\mu = \mu_+(\lambda)$ may well
be below the bisecting line $\mu = \lambda$ for $\lambda$ in a right neighborhood of the origin.}
\end{flushleft}
\label{pianolambdamu}
\end{figure}

\vspace{1 cm}
\normalsize

Authors' addresses:
\bigbreak
\medbreak
\indent Alberto Boscaggin \\
\indent Dipartimento di Matematica, Universit\`a di Torino, \\
\indent Via Carlo Alberto 10, I-10123 Torino, Italy \\
\indent e-mail: alberto.boscaggin@unito.it \\

\medbreak
\indent Maurizio Garrione \\
\indent Dipartimento di Matematica e Applicazioni, Universit\`a di Milano-Bicocca, \\
\indent Via Cozzi 53, I-20125 Milano, Italy \\
\indent e-mail: maurizio.garrione@unimib.it \\

\end{document}